\documentclass[12pt,leqno,twoside,sort]{amsart}

\usepackage{amssymb,amsmath,amsthm,soul,color}

\usepackage[normalem]{ulem}

\usepackage{amssymb,amsmath,amsthm}

\usepackage{mathrsfs}

\usepackage[utf8]{inputenc}

\usepackage{todonotes}

\usepackage[left=2cm, right=2cm, top=2cm, bottom=2cm]{geometry}

\usepackage{url}

\usepackage[numbers,sort&compress]{natbib}

\usepackage{fixmath}

\usepackage[T1]{fontenc}
\usepackage{lmodern}
\usepackage{microtype}
\usepackage[normalem]{ulem}

\usepackage[colorlinks=true,urlcolor=blue,
citecolor=red,linkcolor=blue,linktocpage,pdfpagelabels,
bookmarksnumbered,bookmarksopen]{hyperref}

\usepackage{xcolor,cancel}

\newtheorem{Th}{Theorem}[section]

\newtheorem{Lem}[Th]{Lemma}
\newtheorem{Cor}[Th]{Corollary}

\newtheorem{Rem}[Th]{Remark}

\newcommand{\wt}{\widetilde}

\newcommand{\eps}{\varepsilon}

\newcommand{\R}{\mathbb{R}}

\newcommand{\Z}{\mathbb{Z}}

\newcommand{\cB}{{\mathcal B}}
\newcommand{\cC}{{\mathcal C}}
\newcommand{\cD}{{\mathcal D}}
\newcommand{\cE}{{\mathcal E}}
\newcommand{\cF}{{\mathcal F}}

\newcommand{\cH}{{\mathcal H}}
\newcommand{\cI}{{\mathcal I}}
\newcommand{\cJ}{{\mathcal J}}

\newcommand{\cM}{{\mathcal M}}
\newcommand{\cN}{{\mathcal N}}
\newcommand{\cO}{{\mathcal O}}
\newcommand{\cP}{{\mathcal P}}

\newcommand{\cS}{{\mathcal S}}

\newcommand{\Ga}{\Gamma}

\newcommand{\weakto}{\rightharpoonup}

\newcommand{\curl}{\nabla \times}
\renewcommand{\div}{\mathrm{div}\,}

\newcommand{\spann}{\mathrm{span}\,}

\numberwithin{equation}{section}

\DeclareMathOperator*{\essinf}{ess\,inf}

\begin{document}


\title{Solutions to a nonlinear Maxwell equation with two competing nonlinearities in $\mathbb{R}^3$}

\author[B. Bieganowski]{Bartosz Bieganowski}
\address[B. Bieganowski]{\newline\indent  	Institute of Mathematics,		\newline\indent 	Polish Academy of Sciences, \newline\indent ul. \'Sniadeckich 8, 00-956 Warszawa, Poland
\newline\indent and
\newline\indent  	Faculty of Mathematics and Computer Science,		\newline\indent 	Nicolaus Copernicus University, \newline\indent ul. Chopina 12/18, 87-100 Toru\'n, Poland}
\email{\href{mailto:bbieganowski@impan.pl}{bbieganowski@impan.pl}}		\email{\href{mailto:bartoszb@mat.umk.pl}{bartoszb@mat.umk.pl}}	

\date{}	\date{\today} 

\begin{abstract} 
We are interested in the nonlinear, time-harmonic Maxwell equation
$$
\curl (\curl \mathbf{E} ) + V(x) \mathbf{E} = h(x, \mathbf{E})\mbox{ in } \R^3
$$
with sign-changing nonlinear term $h$, i.e. we assume that $h$ is of the form
$$
h(x, \alpha w) = f(x, \alpha) w - g(x, \alpha) w
$$
for $w \in \R^3$, $|w|=1$ and $\alpha \in \R$. In particular, we can consider the nonlinearity consisting of two competing powers $h(x, \mathbf{E}) = |\mathbf{E}|^{p-2}\mathbf{E} - |\mathbf{E}|^{q-2}\mathbf{E}$ with $2 < q < p < 6$. Under appriopriate assumptions, we show that weak, cylindrically equivariant solutions of the special form are in one-to-one correspondence with weak solutions to a Schr\"odinger equation with a singular potential. Using this equivalence result we show the existence of the least energy solution among cylindrically equivariant solutions of the particular form to the Maxwell equation, as well as to the Schr\"odinger equation.

\medskip

   
\end{abstract}

\maketitle

\section{Introduction}

In the electromagnetism, the behaviour of the electric field $\cE$, magnetic field $\cB$, electric displacement field $\cD$ and magnetic induction $\cH$ is described by the system of Maxwell equations
$$
\left\{ \begin{array}{l}
\curl \cH = \cJ + \frac{\partial \cD}{\partial t} \\
\div (\cD) = \rho \\
\frac{\partial \cB}{\partial t} + \curl \cE = 0 \\
\div (\cB) = 0,
\end{array} \right.
$$
where $\cJ$ denotes the electric current intensity and $\rho$ the electric charge density. We consider these equations in $\R^3$. Then, we can introduce the constitutive relations
$$
\left\{ 
\begin{array}{l}
\cD = \varepsilon \cE + \cP \\
\cH = \frac{1}{\mu} \cB - \cM,
\end{array} \right.
$$
where $\cP, \cM : \R^3 \times \R \rightarrow \mathbb{C}^3$ denote the polarization and the magnetization respectively, and $\varepsilon, \mu : \R^3 \rightarrow \R$ denote the permittivity and permeability of the medium. We are interested in study the electromagnetic waves in the absence of charges, currents and magnetization, i.e. we assume that $\cJ \equiv 0$, $\cM \equiv 0$, $\rho = 0$. Then, the system of Maxwell equations with the constitutive relations lead to
$$
\curl \left( \frac{1}{\mu} \curl \cE \right) + \varepsilon \frac{\partial^2 \cE}{\partial t^2} = - \frac{\partial^2 \cP}{\partial t^2}.
$$
For more physical background see eg. \cite{Dorfler, K, Nie}. Assuming that $\mu \equiv 1$ is constant and looking for time-harmonic fields $\cE = \mathbf{E}(x) e^{i \omega t}$, $\cP = \mathbf{P}(x) e^{i\omega t}$, where $\mathbf{P}$ depends (nonlinearily) on $\mathbf{E}$ lead to the general, time-harmonic Maxwell equation (\textit{curl-curl problem})
\begin{equation}\label{eq:maxwell}
\curl (\curl \mathbf{E} ) + V(x) \mathbf{E} = h(x, \mathbf{E}), \quad x \in \R^3.
\end{equation}
We assume that $h$ is of the form 
\begin{equation}\label{eq:h}
h(x, \alpha w) = f(x, \alpha) w - g(x, \alpha) w
\end{equation}
for all $w \in \partial B(0, 1) \subset \R^3$, $\alpha \in \R$ and a.e. $x \in \R^3$, where $f,g : \R^3 \times \R \rightarrow \R$ satisfy assumptions described below.

Equation \eqref{eq:maxwell} with periodic and sign-changing potentials $V$ has been studied in \cite{BDPR} in cylindrically symmetric setting. On the other hand negative and bounded away from 0 potentials were studied in \cite{Med}. The Maxwell equation in bounded domain $\Omega \subset \R^3$ with the metallic boundary condition
$$
\nu \times \mathbf{E} = 0 \quad \mbox{on } \partial \Omega,
$$
where $\nu : \partial \Omega \rightarrow \R^3$ is the exterior normal vector field, has been studied in a series of papers by Bartsch and Mederski (see \cite{BM1, BM2, BM3}). Their approach is variational and is based on the Helmholtz decomposition and the Nehari-Pankov manifold method. We notice that the problem has been studied also by means of numerical methods (see eg. \cite{Monk}).

Observe that the kernel of $\curl\curl$ has an infinite dimension, and the variational functional associated to \eqref{eq:maxwell} is unbounded from below and from above (in fact, it is strongly indefinite), its critical points have infinite Morse index. Moreover its derivative is not weak-to-weak* continuous. Hence, we do not know whether a limit of a bounded Palais-Smale (or Cerami) sequence is a critical point. The application of a linking type argument in the spirit of \cite{KrSz} is also not immediate. Hence, we will look for \textit{cylindrically equivariant solutions}, which reduces the problem to the Schr\"odinger equation with a singular potential. The reduction is well-known and easy to compute in a case of classical solutions, i.e. looking for classical solutions of the form
\begin{equation}\label{eq:form}
\mathbf{E} (x) = \frac{u(r, x_3)}{r} \left( \begin{array}{c} -x_2 \\ x_1 \\ 0
\end{array} \right), \quad r = \sqrt{x_1^2 + x_2^2}
\end{equation}
we see that $\div \mathbf{E} = 0$ and $u$ satisfies
\begin{equation}\label{eq:schrodinger}
-\Delta u + \frac{u}{r^2} + V(r, x_3) u = f(x,u) - g(x,u),
\end{equation}
where $\Delta = \frac{\partial^2}{\partial r^2} + \frac{1}{r} \frac{\partial}{\partial r} + \frac{\partial^2}{\partial x_3^2}$ is the 3-dimensional Laplace operator in cylindrically symmetric coordinates $(r, x_3)$. We will show this fact also for weak solutions and extend the very recent analysis by Gaczkowski, Mederski, Schino (\cite{GMS}).

\begin{itemize}
\item[(V)] $V \in L^\infty (\R^3)$, $V = V(r, x_3)$ is cylindrically symmetric with $r = \sqrt{x_1^2 + x_2^2}$, $1$-periodic in $x_3$ and 
\begin{equation}\label{spektrum}
\inf \sigma \left(-\Delta + \frac{1}{r^2} + V(r,x_3) \right) = \inf \sigma \left( -\frac{\partial^2}{\partial r^2} - \frac{1}{r} \frac{\partial}{\partial r} - \frac{\partial^2}{\partial x_3^2} + \frac{1}{r^2} + V(r, x_3) \right) > 0.
\end{equation}
\end{itemize}

Note that \eqref{spektrum} is satisfied if $V \in L^\infty (\R^3)$ is cylindrically symmetric, 1-periodic in $x_3$ and e.g. $\essinf_{\R^3} V > 0$. There are also sign-changing potentials satisfying (V). Consider the following function
$$
V(r, x_3) = -\frac{1}{2r^2} \chi_{[1,+\infty)}(r) + \left( \frac14 + \frac18 \cos(2 \pi x_3) \right).
$$
It is clear that $V$ satisfies (V), since $\frac{1}{r^2} + V(r, x_3) \geq \frac{1}{2r^2} + \frac18$. Moreover, for $r = 1$ we get $V(1,x_3) = -\frac14 + \frac18 \cos(2\pi x_3) \leq -\frac18$.

In what follows $O(N)$ denotes the group of real, orthogonal $N\times N$-matrices. We will consider also the group
$$
SO(N) = \{ g \in O(N) \ : \ \det(g) = 1 \}.
$$ 
The action of $O(N)$ (or $SO(N)$) on $\R^N$ is given by the multiplication by matrices, i.e.
$$
O(N) \times \R^N \ni (g, x) \mapsto g x \in \R^N.
$$ 

We assume the following
\begin{enumerate}
\item[(F1)] $f : \R^3 \times \R \rightarrow \R$ is measurable in $x \in \R^3$, $1$-periodic in $x_3$ and continuous in $u \in \R$, $O(2) \times \{ I \}$ invariant in $x \in \R^3$ and there is $2 < p < 6$ such that
$$
|f(x,u)| \leq c (1 + |u|^{p-1}) \mbox{ for all } u \in \R \mbox{ and a.e. } x \in \R^3.
$$
\item[(F2)] $f(x,u) = o(|u|)$ uniformly in $x$ as $u \to 0$.
\item[(F3)] $F(x,u) / |u|^q \to +\infty$ uniformly in $x$ as $|u| \to +\infty$, where $F (x, u) := \int_0^u f(x, s) \, ds$, and $F(x,u) \geq 0$ for all $u \in \R$ and a.e. $x \in \R^3$.
\item[(F4)] $u \mapsto f(x,u)/|u|^{q-1}$ is nondecreasing on $(-\infty, 0)$ and on $(0,+ \infty)$.
\item[(G1)] $g : \R^3 \times \R \rightarrow \R$ is measurable in $x \in \R^3$, $1$-periodic in $x_3$ and continuous in $u \in \R$, $O(2) \times \{ I \}$ invariant in $x \in \R^3$ and there is $2 < q < p$ such that
$$
|g(x,u)| \leq c (1 + |u|^{q-1}) \mbox{ for all } u \in \R \mbox{ and a.e. } x \in \R^3.
$$
\item[(G2)] $g(x,u) = o(|u|)$ uniformly in $x$ as $u \to 0$.
\item[(G3)] $u \mapsto g(x,u)/|u|^{q-1}$ is nonincreasing on $(-\infty, 0)$ and on $(0,+ \infty)$, and there holds 
$$
g(x,u) u \geq 0  \mbox{ for all } u \in \R \mbox{ and a.e. } x \in \R^3.
$$
\end{enumerate}

It is clear that the pure power nonlinearity $g(x,u) = \Ga(r, x_3) |u|^{q-2} u$ with $2 < q < p$ and cylindrically symmetric, positive and bounded away from zero $\Ga \in L^\infty(\R^3)$ satisfies (G1)--(G3). 

Note that (F1) and (F2) imply that for any $\eps > 0$ there is $C_\eps > 0$ such that
\begin{equation}\label{f-eps}
|f(x,u)| \leq \eps |u| + C_\eps |u|^{p-1}.
\end{equation}
Similarly, (G1), (G2) imply the inequality
\begin{equation}\label{g-eps}
|g(x,u)| \leq \eps |u| + C_\eps |u|^{q-1}.
\end{equation}

In what follows we shall denote also $\wt{f}(x,u) := f(x,u) - g(x,u)$ and $\wt{F}(x,u) := F(x,u) - G(x,u)$.

Our first result concerns the correspondence between weak solutions of \eqref{eq:maxwell} and \eqref{eq:schrodinger}.

\begin{Th}\label{Th:1}
Suppose that (V) holds and $\wt{f} : \R^3 \times \R \rightarrow \R$ is measurable and $O(2) \times \{I\}$ invariant in $x \in \R^3$, continuous in $u \in \R$ and satisfies
$$
|\wt{f}(x,u)| \leq c (|u| + |u|^5)\mbox{ for all } u \in \R \mbox{ and a.e. } x \in \R^3.
$$
If $\mathbf{E} \in H^1 (\R^3; \R^3)$ of the form \eqref{eq:form} is, for some cylindrically symmetric $u$, a weak solution to \eqref{eq:maxwell}, then $u \in H^1 (\R^3)$ and $u$ is a weak solution to \eqref{eq:schrodinger}. If $u \in H^1 (\R^3)$ is a cylindrically symmetric, weak solution to \eqref{eq:schrodinger} then $\mathbf{E} \in H^1 (\R^3; \R^3)$ is a weak solution to \eqref{eq:maxwell}, where $\mathbf{E}$ is given by \eqref{eq:form}. Moreover $\div \mathbf{E} = 0$ and $\cE(\mathbf{E}) = \cJ(u)$, where $\cE$ and $\cJ$ are energy functionals defined by \eqref{eq:E} and \eqref{eq:J}, respectively.
\end{Th}

We observe that (F1), (F2), (G1) and (G2) imply that $\wt{f}$ satisfies assumptions in Theorem \ref{Th:1}.

Now we are ready to state our existence results.

\begin{Th}\label{Th:2}
Suppose that (V), (F1)--(F4), (G1)--(G3) hold. Then there exists a cylindrically symmetric, weak solution $u \in H^1 (\R^3)$ to \eqref{eq:schrodinger} being the least energy solution among all cylindrically symmetric solutions with $\int_{\R^3} \frac{u^2}{r^2} \, dx < +\infty$.
\end{Th}

As a consequence of Theorem \ref{Th:1} and Theorem \ref{Th:2} we obtain the following existence result.

\begin{Th}\label{Th:3}
Suppose that (V), (F1)--(F4), (G1)--(G3) hold. Then there exists a weak solution $\mathbf{E} \in H^1 (\R^3; \R^3)$ of \eqref{eq:maxwell} of the form \eqref{eq:form}, for some cylindrically symmetric $u \in H^1 (\R^3)$. Moreover $\mathbf{E}$ is the least energy solution among all solutions of the form \eqref{eq:form}.
\end{Th}

\begin{Rem}
If we take $q = 2$ and $g \equiv 0$, statements of Theorem \ref{Th:2} and \ref{Th:3} still hold true under (V), (F1)--(F4) and proofs require only slight modifications.
\end{Rem}

\begin{Rem}
The sign-changing behaviour of the right hand side of the equation forces us to consider the positive definite case in the singular Schr\"odinger problem \eqref{eq:schrodinger} (see the assumption (V)). It would be interesting to find least energy solutions to a strongly indefinite problem with sign-changing and periodic (with respect to some of $x$ arguments) nonlinearity. Strongly indefinite problems with sign-changing nonlinearities has been studied eg. in \cite{GuZhou} with the nonlinear term in the $L^p$-space with respect to $x$, where the multiplicity of solutions has been estabilished. The case of a negative potential is physically motivated, since $V$ is of the form $V(x) = - \omega^2 \varepsilon(x)$.
\end{Rem}

In the Appendix we provide the sketch of the proof of the multiplicity of solutions to the Schr\"odinger equation \eqref{eq:schrodinger}. Theorem \ref{Th:1} allows us then to obtain the multiplicity for the curl-curl problem \eqref{eq:maxwell}, see Theorems \ref{Th:4} and \ref{Th:5} in the Appendix.

\section{Functional setting for the Schr\"odinger equation (\ref{eq:schrodinger})}

We introduce the space
$$
X := \left\{ u \in H^1 (\R^3) \ : \  u = u(r, x_3) \mbox{ is cylindrically symmetric and } \int_{\R^3} \frac{|u(r, x_3)|^2}{r^2} \, dx < +\infty \right\}
$$
endowed with the norm
$$
\| u\|_X^2 := \int_{\R^3} |\nabla u|^2 + \frac{u^2}{r^2} + u^2 \, dx, \mbox{ where } r = \sqrt{x_1^2 + x_2^2}.
$$
It is known that $(X, \| \cdot \|_X)$ is a Hilbert space. Note that (V) implies that the quadratic form
$$
X \ni u \mapsto Q(u) := \int_{\R^3} |\nabla u|^2 + \frac{u^2}{r^2} + V(r, x_3) u^2 \, dx \in \R
$$
induces a norm $\| u\|^2 := Q(u)$ on $X$ which is equivalent with $\| \cdot \|_X$. It is clear that the embedding $X \hookrightarrow H^1 (\R^3)$ is continuous and therefore Sobolev embeddings $X \hookrightarrow L^t (\R^3)$ for $t \in [2,6]$ are continuous.

We define the energy functional associated to \eqref{eq:schrodinger} $\cJ : X \rightarrow \R$ by the formula
\begin{equation}\label{eq:J}
\cJ(u) := \frac12 \int_{\R^3} |\nabla u|^2 + \frac{u^2}{r^2} + V(r, x_3) u^2 \, dx - \int_{\R^3} F(x,u) \, dx + \int_{\R^3} G(x,u) \, dx, \quad u \in X.
\end{equation}
One can easily check that under (F1), (G1), $\cJ$ is of $\cC^1$ class on $X$. We say that critical points of $\cJ$ are \textit{weak solutions} to \eqref{eq:schrodinger}. In our setting we may rewrite $\cJ$ in the following form
$$
\cJ(u) = \frac12 \|u\|^2 - \int_{\R^3} F(x,u) \, dx + \int_{\R^3} G(x,u) \, dx, \quad u \in X.
$$

The scalar product $\langle \cdot, \cdot \rangle$ on $X$ is given by
$$
\langle u, v \rangle := \int_{\R^3} \nabla u \cdot \nabla v + \frac{u v}{r^2} + V(r, x_3) u v \, dx,
$$
where $\cdot$ denotes the usual scalar product in $\R^3$. Hence
$$
\cJ'(u)(v) = \langle u, v \rangle - \int_{\R^3} f(x,u)v \, dx + \int_{\R^3} g(x,u)v \, dx, \quad u, v \in X.
$$
Since nontrivial weak solutions to \eqref{eq:schrodinger} are critical point of $\cJ$ it is clear that they lie in the so-called \textit{Nehari manifold}
$$
\cN := \{ u \in X \setminus \{0\} \ : \ \cJ'(u)(u) = 0 \},
$$
which is, under our assumptions, a topological manifold (not necessarily a differentiable manifold). Observe that $\cC_0^\infty (\R^3) \not\subset X$ and $\cJ'(u)(\varphi)$ is not necessarily finite for $u \in X$ and $\varphi \in \cC_0^\infty (\R^3)$. Hence we say that $u \in X$ is a weak solution to \eqref{eq:schrodinger} if $u$ is a critical point of $\cJ$.

Moreover, it is classical to check that $\cJ'$ is weak-to-weak* continuous, i.e. for any $(u_n) \subset X$ with $u_n \weakto u_0$ in $X$ and $v \in X$ there holds
$$
\cJ'(u_n)(v) \to \cJ'(u_0)(v).
$$
Hence, if $(u_n) \subset X$ is a sequence with $\cJ'(u_n) \to 0$ in $X^*$, then any weak limit point of $(u_n)$ is a critical point of $\cJ$.

\section{Functional setting for the Maxwell equation (\ref{eq:maxwell})}

We introduce the energy functional $\cE : H^1 (\R^3; \R^3) \rightarrow \R$ associated with \eqref{eq:maxwell}
\begin{equation}\label{eq:E}
\cE (\mathbf{E}) = \frac12 \int_{\R^3} |\curl \mathbf{E}|^2 \, dx + \frac12 \int_{\R^3} V(x) |\mathbf{E}|^2 \, dx - \int_{\R^3} H(x, \mathbf{E}) \, dx,
\end{equation}
where 
$$
H(x, \mathbf{E}) := \int_0^1 h(x, t \mathbf{E}) \cdot \mathbf{E}  \, dt,
$$
$h$ is given by \eqref{eq:h} and $\cdot$ denotes the usual scalar product in $\R^3$. $\cE$ is of $\cC^1$ class on $H^1 (\R^3; \R^3)$ and we say that its critical points are \textit{weak solutions} to \eqref{eq:maxwell}. We observe also that for any $\mathbf{E} \in H^1 (\R^3; \R^3)$ there holds
\begin{equation}\label{eq:curl-grad}
\int_{\R^3} |\curl \mathbf{E}|^2 + |\div \mathbf{E}|^2 \, dx = \int_{\R^3} |\nabla \mathbf{E}|^2 \, dx,
\end{equation}
where $\div\mathbf{E}$ and $\nabla \mathbf{E}$ denotes the distributional divergence and the distributional gradient, respectively. Consider the action of the group 
$$
SO := SO(2) \times \{I\} = \left\{ \left( \begin{array}{ccc}
\cos \alpha & \sin \alpha & 0 \\
-\sin \alpha & \cos \alpha & 0 \\
0 & 0 & 1
\end{array} \right) \ : \ \alpha \in [0, 2\pi) \right\}
$$ 
on $\R^3$. Introduce the set
$$
\cD := \left\{ \mathbf{E} \in H^1 (\R^3; \R^3) \ : \ \mathbf{E} \mbox{ is of the form } \eqref{eq:form} \mbox{ for some } SO\mbox{-invariant } u : \R^3 \rightarrow \R \right\}.
$$
Obviously $\cD$ is closed in $H^1 (\R^3; \R^3)$ and $\cD \subset H^1_{SO-\mathrm{equiv}} (\R^3; \R^3)$, where
$$
H^1_{SO-\mathrm{equiv}} (\R^3; \R^3) := \left\{ \mathbf{E} \in H^1 (\R^3; \R^3) \ : \ \mathbf{E} \mbox{ is } SO\mbox{-equivariant} \right\}.
$$
We note the following density result.

\begin{Lem}
Let $\mathbf{E} \in H^1_{SO-\mathrm{equiv}} (\R^3; \R^3)$. Then there exists a sequence $(\mathbf{E}_n) \subset \cC_0^\infty (\R^3; \R^3) \cap H^1_{SO-\mathrm{equiv}} (\R^3; \R^3)$ such that $\| \mathbf{E}_n - \mathbf{E} \|_{H^1 (\R^3; \R^3)} \to 0$.
\end{Lem}

\begin{proof}
Fix $\mathbf{E} \in H^1_{SO-\mathrm{equiv}} (\R^3; \R^3) \subset H^1 (\R^3; \R^3)$. Obviously, there is a sequence $(\mathbf{U}_n) \subset \cC_0^\infty (\R^3; \R^3)$ such that $\| \mathbf{U}_n - \mathbf{E} \|_{H^1 (\R^3; \R^3)} \to 0$. Put
$$
\mathbf{E}_n (x) := \int_{SO} g^{-1} \mathbf{U}_n(gx) \, d \mu (g),
$$
where $\mu$ is the (probabilistic) Haar measure on a compact group $SO$, in particular $\mu(SO)=1$. Obviously $\mathbf{E}_n \in \cC_0^\infty (\R^3; \R^3)$. Moreover, for any $e \in SO$
$$
\mathbf{E}_n (ex) = \int_{SO} g^{-1} \mathbf{U}_n(gex) \, d \mu(g) = \int_{SO} \left( \wt{g} e^{-1} \right)^{-1} \mathbf{U}_n( \wt{g} x) \, d \mu(\wt{g}) = e \int_{SO} \wt{g}^{-1} \mathbf{U}_n(\wt{g}x) \, d \mu (\wt{g}) = e \mathbf{E}_n (x)
$$
and $\mathbf{E}_n \in \cC_0^\infty (\R^3; \R^3) \cap H^1_{SO-\mathrm{equiv}} (\R^3; \R^3)$. For any $m,n$ we compute
\begin{align*}
| \mathbf{E}_n - \mathbf{E}_m |_2^2 &= \int_{\R^3} \left| \int_{SO} g^{-1} (\mathbf{U}_n(gx) - \mathbf{U}_m(gx) ) \, d\mu(g) \right|^2 \, dx \\
&\leq \int_{\R^3}  \int_{SO} \left| g^{-1} (\mathbf{U}_n(gx) - \mathbf{U}_m(gx) ) \right|^2 \, d\mu(g)  \, dx \\
&= \int_{SO} \int_{\R^3}   \left| g^{T} ( \mathbf{U}_n(gx) -  \mathbf{U}_m(gx) ) \right|^2  \, dx \, d\mu(g) \\
&= \int_{SO} \int_{\R^3}   \left|  \mathbf{U}_n(gx) -  \mathbf{U}_m(gx)  \right|^2  \, dx \, d\mu(g) = \int_{SO} | \mathbf{U}_n -  \mathbf{U}_m |_2^2 \, d \mu(g) = | \mathbf{U}_n -  \mathbf{U}_m |_2^2,
\end{align*}
since $\det (g) = 1$ and $\mu(SO) = 1$. Similarly
$$
| \nabla \left( \mathbf{E}_n - \mathbf{E}_m \right) |_2^2 \leq | \nabla \left( \mathbf{U}_n - \mathbf{U}_m \right) |_2^2.
$$
Fix $\varepsilon > 0$. Then, for sufficiently large $m, n$
\begin{align*}
\| \mathbf{E}_n - \mathbf{E}_m \|_{H^1 (\R^3; \R^3)}^2 = | \nabla \left( \mathbf{E}_n - \mathbf{E}_m \right) |_2^2 + | \mathbf{E}_n - \mathbf{E}_m |_2^2 \leq | \nabla \left( \mathbf{U}_n - \mathbf{U}_m \right) |_2^2 + | \mathbf{U}_n - \mathbf{U}_m |_2^2 < \varepsilon^2,
\end{align*}
since $(\mathbf{U}_n)$ is a Cauchy sequence in $H^1 (\R^3; \R^3)$. Hence $(\mathbf{E}_n)$ is also a Cauchy sequence in $H^1 (\R^3; \R^3)$, so it is convergent and it is sufficient to show that its limit is $\mathbf{E}$. Obviously, due to the continuous Sobolev embedding $H^1 (\R^3; \R^3) \subset L^2 (\R^3; \R^3)$, it is enough to show that $\mathbf{E}_n \to \mathbf{E}$ in $L^2 (\R^3; \R^3)$. We perform the computatons similarly as above:
\begin{align*}
| \mathbf{E}_n - \mathbf{E} |_2^2 &= \int_{\R^3} \left| \int_{SO} g^{-1} (\mathbf{U}_n(gx) - g\mathbf{E}(x) ) \, d\mu(g) \right|^2 \, dx \\
&\leq \int_{\R^3}  \int_{SO} \left| g^{-1} (\mathbf{U}_n(gx) - \mathbf{E}(gx) ) \right|^2 \, d\mu(g)  \, dx \\
&= \int_{SO} \int_{\R^3}   \left| g^{-1} ( \mathbf{U}_n(gx) -  \mathbf{E}(gx) ) \right|^2  \, dx \, d\mu(g) \\
&= \int_{SO} \int_{\R^3}   \left|  \mathbf{U}_n(gx) -  \mathbf{E}(gx)  \right|^2  \, dx \, d\mu(g) = \int_{SO} | \mathbf{U}_n -  \mathbf{E} |_2^2 \, d \mu(g) = | \mathbf{U}_n -  \mathbf{E} |_2^2 \to 0.
\end{align*}
\end{proof}

Let
$$
\cH := \left\{ \mathbf{E} \in L^2 (\R^3; \R^3) \ : \ |\mathbf{E}(x)| \leq C r \mbox{ for some } C > 0 \mbox{ uniformly w.r. to } x_3 \mbox{ as } r \to 0^+ \right\},
$$
where $r = \sqrt{x_1^2 + x_2^2}$, and
$$
\R^3_* := \R^3 \setminus ( \{0\} \times \{0\} \times \R).
$$
We obtain the following characterization of $\cD$.

\begin{Lem}\label{lem:D}
There holds 
$$
\cD = \overline{ \cC_0 (\R^3; \R^3) \cap \cC^\infty (\R^3_*; \R^3) \cap \cH \cap \cD },
$$
where the closure is taken with respect to the $H^1 (\R^3; \R^3)$ norm.
\end{Lem}

\begin{proof}
Put $A := \overline{ \cC_0 (\R^3; \R^3) \cap \cC^\infty (\R^3_*; \R^3) \cap \cH \cap \cD }$. Since $\cD$ is closed in $H^1 (\R^3; \R^3)$ it is clear that
$$
A \subset \overline{\cD} = \cD.
$$
Fix $\mathbf{E} \in \cD$. Let $(\mathbf{E}_n) \subset \cC_0^\infty (\R^3; \R^3) \cap H^1_{SO-\mathrm{equiv}} (\R^3; \R^3)$ be a sequence with $\| \mathbf{E}_n - \mathbf{E} \|_{H^1 (\R^3; \R^3)} \to 0$. We will write $\mathbf{E}_n = (\mathbf{E}_n^1, \mathbf{E}_n^2, \mathbf{E}_n^3)$. For $x \in \R^3_*$ define
\begin{align*}
&\mathbb{V}_\rho (x) := \spann \{ (x_1, x_2, 0) \}, \\
&\mathbb{V}_\tau (x) := \spann \{ (-x_2, x_1, 0) \}, \\
&\mathbb{V}_\zeta (x) := \spann \{ (0,0,1) \}.
\end{align*}
Note that $\mathbb{V}_\rho(x), \mathbb{V}_\tau (x), \mathbb{V}_\zeta (x)$ are one-dimensional subspaces of $\R^3$ and
$$
\R^3 = \mathbb{V}_\rho(x) \oplus \mathbb{V}_\tau (x) \oplus \mathbb{V}_\zeta (x)
$$
for any $x \in \R^3_*$.
Let $\cD^{1,2} (\R^3; \R^3)$ denote the completion of $\cC_0^\infty (\R^3; \R^3)$ with respect to the norm $\|\mathbf{E}\|_\nabla := |\nabla \mathbf{E}|_2$. Then, as in \cite[Lemma 1]{Azzolini} and \cite[Proposition 2.3]{GMS}, for any $n \geq 1$ there are $SO$-equivariant functions $\mathbf{P}_{\rho,n}, \mathbf{P}_{\tau, n}, \mathbf{P}_{\zeta, n} \in \cD^{1,2} (\R^3; \R^3)$ such that for $x \in \R^3_*$
\begin{itemize}
\item[(i)] $\mathbf{P}_{\rho, n} (x)$ is the projection of $\mathbf{E}_n (x)$ onto $\mathbb{V}_\rho (x)$,
\item[(ii)] $\mathbf{P}_{\tau, n} (x)$ is the projection of $\mathbf{E}_n (x)$ onto $\mathbb{V}_\tau (x)$,
\item[(iii)] $\mathbf{P}_{\zeta, n} (x) := (0, 0, \mathbf{E}_n^3 (x) )$ is the projection of $\mathbf{E}_n(x)$ onto $\mathbb{V}_\zeta (x)$.
\end{itemize}
Hence $\mathbf{P}_{\rho,n}, \mathbf{P}_{\tau, n}, \mathbf{P}_{\zeta, n} \in  \cC^\infty (\R^3_*; \R^3)$ and they vanish outside a sufficiently large ball in $\R^3$. In fact, $\mathbf{P}_{\zeta, n} \in \cC_0^\infty (\R^3; \R^3)$. Moreover
$$
\mathbf{E}_n (x) = \mathbf{P}_{\rho, n}(x) + \mathbf{P}_{\tau, n} (x) + \mathbf{P}_{\zeta, n} (x) \mbox{ for every } x \in \R^3_*.
$$

We notice that, as in the proof of \cite[Proposition 2.3]{GMS}, $\mathbf{P}_{\tau, n} \to \mathbf{E}$ in $\cD^{1,2} (\R^3; \R^3)$. Hence it is sufficient to show that $\mathbf{P}_{\tau, n} \in L^2 (\R^3; \R^3)$, $\mathbf{P}_{\tau,n} \to \mathbf{E}$ in $L^2 (\R^3; \R^3)$ and that $\mathbf{P}_{\tau, n} \in \cC_0 (\R^3; \R^3) \cap \cH$.

Observe that $\mathbf{E}_n^1 = \mathbf{E}_n^2 \equiv 0$ on $\{0\}\times \{0\} \times \R$, since $\mathbf{E}_n$ is $SO$-invariant. Note that for any $x \in \R^3_*$ we can write down exact formulas for $\mathbf{P}_{\rho, n}$ and $\mathbf{P}_{\tau, n}$, i.e.
$$
\mathbf{P}_{\rho, n} (x) = \frac{\mathbf{E}_n(x) \cdot (x_1, x_2, 0)}{r^2} \left( \begin{array}{c}
x_1 \\ x_2 \\ 0
\end{array}  \right)
$$
and
\begin{equation}\label{projection}
\mathbf{P}_{\tau, n} (x) = \frac{\mathbf{E}_n(x) \cdot (-x_2, x_1, 0)}{r^2} \left( \begin{array}{c}
-x_2 \\ x_1 \\ 0
\end{array}  \right).
\end{equation}
From the uniform continuity of $\mathbf{E}_n$ we see that
$$
\lim_{ (x_1, x_2) \to (0,0)} \mathbf{P}_{\rho,n}(x) = \lim_{ (x_1, x_2) \to (0,0)} \mathbf{P}_{\tau,n}(x) = 0
$$
uniformly with respect to $x_3$. Hence we can extend $\mathbf{P}_{\rho, n}$ and $\mathbf{P}_{\tau, n}$ continuously onto $\R^3$ and write that
$$
\mathbf{E}_n (x) = \mathbf{P}_{\rho, n}(x) + \mathbf{P}_{\tau, n} (x) + \mathbf{P}_{\zeta, n} (x) \mbox{ for every } x \in \R^3.
$$
In particular $\mathbf{P}_{\rho, n}, \mathbf{P}_{\tau, n} \in \cC_0 (\R^3; \R^3)$ and $\mathbf{P}_{\tau, n} \in L^2 (\R^3; \R^3)$. 

Moreover
\begin{align*}
| \mathbf{P}_{\tau, n} - \mathbf{E} |_2 \leq | \mathbf{P}_{\tau, n} - \mathbf{E}_n|_2 + |\mathbf{E}_n - \mathbf{E}|_2 = | \mathbf{P}_{\tau, n} - \mathbf{E}_n|_2 + o(1)
\end{align*}
and, recalling \eqref{projection},
\begin{align*}
| \mathbf{P}^1_{\tau, n} - \mathbf{E}^1_n|_2^2 &= \int_{\R^3} \left| \frac{\mathbf{E}_n(x) \cdot (-x_2, x_1, 0)}{r^2} (-x_2) - \mathbf{E}^1_n (x) \right|^2 \, dx \\
&= \int_{\R^3} \left| \frac{ x_2^2 \mathbf{E}_n^1 (x)  - x_1 x_2 \mathbf{E}_n^2 (x) - x_1^2 \mathbf{E}_n^1(x)  - x_2^2 \mathbf{E}_n^1(x) }{x_1^2 + x_2^2} \right|^2 \, dx \\
&= \int_{\R^3} \left| \frac{x_1 x_2 \mathbf{E}_n^2 (x) + x_1^2 \mathbf{E}_n^1 (x) }{x_1^2 +  x_2^2} \right|^2 \, dx.
\end{align*}
Direct calculations show that, taking into account \eqref{eq:form},
$$
\frac{x_1 x_2 \mathbf{E}^2 (x) + x_1^2 \mathbf{E}^1 (x) }{x_1^2 +  x_2^2} = 0\mbox{ for a.e. } x \in \R^3_*.
$$
Hence
\begin{align*}
| \mathbf{P}^1_{\tau, n} - \mathbf{E}^1_n|_2 &= \left( \int_{\R^3} \left| \frac{x_1 x_2 \mathbf{E}_n^2 (x) + x_1^2 \mathbf{E}_n^1 (x) }{x_1^2 +  x_2^2} - \frac{x_1 x_2 \mathbf{E}^2 (x) + x_1^2 \mathbf{E}^1 (x) }{x_1^2 +  x_2^2} \right|^2 \, dx \right)^{1/2} \\
&= \left( \int_{\R^3} \left| \frac{x_1 x_2 ( \mathbf{E}_n^2 (x) - \mathbf{E}^2(x) ) + x_1^2 (\mathbf{E}_n^1 (x) - \mathbf{E}^1 (x) )}{x_1^2 + x_2^2} \right|^2 \, dx \right)^{1/2} \\
&\leq \left(  \int_{\R^3} \left| \frac{x_1 x_2 ( \mathbf{E}_n^2 (x) - \mathbf{E}^2(x) ) }{x_1^2+x_2^2} \right|^2 \, dx \right)^{1/2} + \left( \int_{\R^3} \left| \frac{x_1^2 (\mathbf{E}_n^1 (x) - \mathbf{E}^1 (x) )}{x_1^2 + x_2^2} \right|^2 \, dx \right)^{1/2} \\
&=  \left(  \int_{\R^3} \left| \frac{x_1 x_2 }{x_1^2+x_2^2} \right|^2 |\mathbf{E}_n^2 (x) - \mathbf{E}^2(x)|^2 \, dx \right)^{1/2} + \left( \int_{\R^3} \left| \frac{x_1^2}{x_1^2 + x_2^2} \right|^2 | \mathbf{E}_n^1 (x) - \mathbf{E}^1 (x) |^2 \, dx \right)^{1/2} \\
&\leq \frac12 | \mathbf{E}_n^2 - \mathbf{E}^2 |_2 + | \mathbf{E}_n^1 - \mathbf{E}^1 |_2 \to 0.
\end{align*}
Similarly
$$
| \mathbf{P}^2_{\tau, n} - \mathbf{E}^2_n|_2 \to 0, \quad | \mathbf{P}^3_{\tau, n} - \mathbf{E}^3_n|_2 \to 0.
$$
Hence $\mathbf{P}_{\tau, n} \to \mathbf{E}$ in $L^2 (\R^3; \R^3)$.

To show that $\mathbf{P}_{\tau, n} \in \cH$ we note that $\mathbf{E}_n - \mathbf{P}_{\zeta, n} \in \cC_0^\infty (\R^3; \R^3)$. Then, by Taylor-series expansion
\begin{align*}
(\mathbf{P}_{\rho,n} + \mathbf{P}_{\tau, n})(x) &= (\mathbf{E}_n - \mathbf{P}_{\zeta, n}) (x) \\
&= \underbrace{(\mathbf{E}_n - \mathbf{P}_{\zeta, n}) (0,0, x_3)}_{=0} + \nabla (\mathbf{E}_n - \mathbf{P}_{\zeta, n}) (0,0,x_3) \left( \begin{array}{c}
x_1 \\ x_2 \\ 0
\end{array} \right) + o (|(x_1, x_2)|) \\
&=  \nabla (\mathbf{E}_n - \mathbf{P}_{\zeta, n}) (0,0,x_3) \left( \begin{array}{c}
x_1 \\ x_2 \\ 0
\end{array} \right) + o (|(x_1, x_2)|)\mbox{ as } |(x_1, x_2)| \to 0^+.
\end{align*}
Hence $\mathbf{P}_{\rho,n} + \mathbf{P}_{\tau, n} \in \cH$. Moreover
$$
| \mathbf{P}_{\tau, n}  | \leq |\mathbf{P}_{\rho,n} + \mathbf{P}_{\tau, n} |,
$$
hence $\mathbf{P}_{\tau, n} \in \cH$ and therefore $\mathbf{E} \in A$, and the proof is completed.
\end{proof}

\section{The equivalence result}

We shall begin with the equivalence of functionals $\cE$ and $\cJ$ for vector fields of the form \eqref{eq:form}.

\begin{Lem}\label{lem:4.1}
$\mathbf{E} \in \cD$ if and only if $u \in X$. Moreover $\cJ(u) = \cE (\mathbf{E})$ and $\div(\mathbf{E}) = 0$.
\end{Lem}

\begin{proof}
Fix $u \in X$ and let $\mathbf{E}$ be given by \eqref{eq:form}. The same argument as in \cite[Lemma 2.4]{GMS} show that the pointwise gradient of $\mathbf{E}$ in $\R^3$ is also the distributional gradient in $\R^3$. Then it is clear that $\mathbf{E} \in H^1 (\R^3; \R^3)$ and therefore $\mathbf{E} \in \cD$. Moreover it is clear that $\div(\mathbf{E}) = 0$ and
$$
\int_{\R^3} | \curl \mathbf{E} |^2 \, dx = \int_{\R^3} |\nabla u|^2 \, dx.
$$
It is also obvious, from \eqref{eq:h} that 
$$
\int_{\R^3} H(x, \mathbf{E}) \, dx = \int_{\R^3} \tilde{F}(x,u) \, dx.
$$
Moreover
$$
\int_{\R^3} V(x) |\mathbf{E}|^2 \, dx = \int_{\R^3} V(x) \frac{u^2}{r^2} \left| \left( \begin{array}{c}
-x_2 \\ x_1 \\ 0
\end{array} \right) \right|^2 \, dx = \int_{\R^3} V(x) u^2 \, dx.
$$
Thus $\cE (\mathbf{E}) = \cJ(u)$. 

On the other hand, fix $\mathbf{E} \in \cD$, and let $u$ be $SO$-invariant map such that \eqref{eq:form} holds. From \eqref{lem:D} there is a sequence $(\mathbf{U}_n) \subset \cC_0 (\R^3; \R^3) \cap \cC^\infty (\R^3_*; \R^3) \cap \cH \cap \cD$ with
$$
|\nabla \mathbf{U}_n - \nabla \mathbf{E}|_2 + |\mathbf{U}_n - \mathbf{E}|_2 \to 0.
$$ 
Moreover $\mathbf{U}_n$ are of the form \eqref{eq:form}, i.e.
$$
\mathbf{U}_n (x) = \frac{u_n(r, x_3)}{r^2} \left( \begin{array}{c}
-x_2 \\ x_1 \\ 0
\end{array} \right),
$$
where $u_n$ are $SO$-equivariant maps. We will show that $u_n \in X$. Note that $|\mathbf{U}_n|^2 = |u_n|^2$ and therefore $u_n \in L^2 (\R^3)$. Since $u_n \in \cC_0 (\R^3)\cap \cC^\infty (\R^3_*)$ and $|u_n(x)| \leq Cr$ for some $C > 0$ as $r \to 0^+$, uniformly with respect to $x_3$, we have
$$
\int_{\R^3} \frac{|u_n|^2}{r^2} \, dx < +\infty.
$$
The same computation as in \cite[Lemma 2.4]{GMS} shows that $\nabla u_n \in L^2 (\R^3; \R^3)$, and therefore $u_n \in X$. It is clear that
$$
\lim_{n \to +\infty} | u_n - u |_2 = \lim_{n\to +\infty} | \mathbf{U}_n - \mathbf{E} |_2 = 0.
$$
Hence it is sufficient to show that $(u_n)$ is a Cauchy sequence in $X$. Fix $\varepsilon > 0$. We have that
\begin{align*}
\| u_n - u_m \|^2 &= \int_{\R^3} | \nabla (u_n -  u_m) |^2 + \frac{(u_n - u_m)^2}{r^2} + V(x) (u_n - u_m)^2 \, dx \\
&\leq \int_{\R^3} |\nabla (\mathbf{U}_n - \mathbf{U}_m)|^2 \, dx + |V|_\infty \int_{\R^3} | \mathbf{U}_n - \mathbf{U}_m |^2 \, dx \\ &=  |\nabla \mathbf{U}_n - \nabla \mathbf{U}_m|_2^2 + |V|_\infty |\mathbf{U}_n - \mathbf{U}_m|_2^2 < \varepsilon^2
\end{align*}
for sufficiently large $n,m$. Recalling that $\div \mathbf{U}_n(x) = 0$ for $x \in \R^3_*$, we easily see that $\div \mathbf{E} = 0$, where the divergence is taken in the distributional sense.
\end{proof}

Now we are ready to show the equivalence of weak solutions. 

\begin{proof}[Proof of Theorem \ref{Th:1}]
Suppose that $\mathbf{E} \in \cD$ is a weak solution to \eqref{eq:maxwell} and $u$ is a $SO$-invariant function satisfying \eqref{eq:form}. 

We recall that, from the Palais principle of criticality and the invariance of $\cE$ with respect to the action
$$
\mathscr{S} \mathbf{E} = -\mathbf{P}_\rho + \mathbf{P}_\tau - \mathbf{P}_\zeta,
$$ 
where $\mathbf{P}_i (x)$ is the projection of $\mathbf{E}(x)$ onto $\mathbb{V}_i (x)$, where $i \in \{ \rho, \tau, \zeta \}$, with $\mathbb{V}_i$ given in Lemma \ref{lem:D}, there follows that it is sufficient to take test functions from $\cD$ (see \cite[Proposition 1]{Azzolini} and \cite[Proof of Theorem 2.1]{GMS}).

Take any $\mathbf{V} \in \cD$ with $SO$-invariant $v$ satisfying \eqref{eq:form}. Then, arguing as in Lemma \ref{lem:4.1}, we obtain that
$$
\int_{\R^3} \nabla \mathbf{E} \cdot \nabla \mathbf{V} \, dx = \int_{\R^3} \nabla u \cdot \nabla v + \frac{uv}{r^2} \, dx.
$$
and
$$
\int_{\R^3} V(x) \mathbf{E} \cdot \mathbf{V} \, dx = \int_{\R^3} V(x) uv \, dx.
$$
Moreover, recalling \eqref{eq:h}, we get
\begin{align*}
\int_{\R^3} h(x, \mathbf{E}) \cdot \mathbf{V} \, dx &= \int_{\R^3} h \left(x, \frac{u (r, x_3)}{r} \left( \begin{array}{c}
-x_2 \\ x_1 \\ 0
\end{array} \right) \right) \cdot \frac{v(r, x_3)}{r} \left( \begin{array}{c}
-x_2 \\ x_1 \\ 0
\end{array} \right) \, dx \\
&= \int_{\R^3} \frac{1}{r^2} \wt{f}(x, u(r, x_3)) v(r, x_3) \left|  \left( \begin{array}{c}
-x_2 \\ x_1 \\ 0
\end{array} \right) \right|^2 \, dx \\
&= \int_{\R^3}\frac{1}{r^2} \wt{f}(x, u(r, x_3)) v(r, x_3) r^2 \, dx = \int_{\R^3} \wt{f}(x, u) v\, dx.
\end{align*}
Hence
$$
\int_{\R^3}  \nabla u \cdot \nabla v + \frac{uv}{r^2} + V(x) uv \, dx - \int_{\R^3} \wt{f}(x,u)v \, dx = 0
$$
for any $v \in X$, due to Lemma \ref{lem:4.1}. Hence $u \in X$ is a weak solution to \eqref{eq:schrodinger}. The same computation shows that if $u \in X$ is a weak solution to \eqref{eq:schrodinger}, then $\mathbf{E}$ given by \eqref{eq:form} is a weak solution to \eqref{eq:maxwell}.
\end{proof}

\section{Critical point theory}

Suppose that $(E, \| \cdot \|)$ is a Hilbert space and $\cJ : E \rightarrow \R$ is a nonlinear functional of the general form
$$
\cJ(u) = \frac12 \|u\|^2 - \cI(u),
$$
where $\cI$ is of $\cC^1$ class and $\cI(0)=0$. We introduce the following set
$$
\cN := \{ u \in E \setminus \{ 0 \} \ : \ \cJ'(u)(u) = 0 \}.
$$

\begin{Th}\label{abstract}
Suppose that
\begin{itemize}
\item[(J1)] there is $r > 0$ such that
$$
\inf_{\|u\|=r} \cJ(u) > 0;
$$
\item[(J2)] $\frac{\cI (t_n u_n)}{t_n^2} \to +\infty$ for $t_n \to +\infty$ and $u_n \to u \neq 0$;
\item[(J3)] for all $t > 0$ and $u \in \cN$ there holds
$$
\frac{t^2-1}{2} \cI'(u)(u) - \cI(tu) + \cI(u) \leq 0.
$$
\end{itemize}
Then $\Ga \neq \emptyset$, $\cN \neq \emptyset$ and
$$
c := \inf_{\cN} \cJ = \inf_{\gamma \in \Gamma} \sup_{t \in [0,1]} \cJ(\gamma(t)) = \inf_{u \in E \setminus \{0\}} \sup_{t \geq 0} \cJ(tu) > 0,
$$
where
$$
\Gamma := \{ \gamma \in \cC ([0,1], E) \ : \ \gamma(0) = 0, \ \|\gamma(1)\| > r, \ \cJ(\gamma(1)) < 0 \}.
$$
Moreover there is a Cerami sequence for $\cJ$ on the level $c$, i.e. a sequence $\{ u_n \}_n \subset E$ such that
$$
\cJ(u_n) \to c, \quad (1+\|u_n\|) \cJ'(u_n) \to 0.
$$
\end{Th}

The foregoing theorem can be shown similarly as in \cite{BM, MSSz}; hovewer we do not require the inequality $\cI(u) \geq 0$ (which does not need to be satisfied in our setting). See also \cite[Theorem 2.1]{BM-cpaa}.

\begin{proof}
Observe that there exists $v \in E \setminus \{0\}$ with $\|v\| > r$ such that $\cJ(v) < 0$. Indeed, fix $u \in E \setminus \{0\}$ and from (J2) there follows that
\begin{equation}\label{infty}
\frac{\cJ(tu)}{t^2} = \frac12 \|u\|^2 - \frac{\cI(tu)}{t^2} \to - \infty \quad \mbox{as } t \to +\infty
\end{equation}
and we may take $v := t u$ for sufficiently large $t > 0$. In particular, the family of paths $\Gamma$ is nonempty. Moreover, $\cJ(tu) \to 0$ as $t \to 0^+$ and for $t = \frac{r}{\|u\|} > 0$ we get $\cJ(tu) > 0$. Hence, taking \eqref{infty} into account, $(0,+\infty) \ni t \mapsto \cJ(tu) \in \R$ has a local maximum, which is a critical point of $\cJ(tu)$ and $tu \in \cN$. Hence $\cN \neq \emptyset$. Suppose that $u \in \cN$. Then, from (J3),
$$
\cJ(tu) = \cJ(tu) - \frac{t^2-1}{2} \cJ'(u)(u) \leq \cJ(u)
$$
and therefore $u$ is a maximizer (not necessarily unique) of $\cJ$ on $\R_+ u := \{ su \ : \ s > 0 \}$. Hence, for any $u \in \cN$ there are $0 < t_{\min} (u) \leq 1 \leq t_{\max}(u)$ such that $t u \in \cN$ for any $t \in [t_{\min}(u), t_{\max} (u)]$ and 
$$
[t_{\min}(u), t_{\max} (u)] \ni t \mapsto \cJ(tu) \in \R
$$
is constant. Moreover $\cJ'(tu)(u) > 0$ for $t \in (0, t_{\min}(u))$ and $\cJ'(tu)(u) < 0$ for $t \in (t_{\max} (u), +\infty)$, $E \setminus \cN$ consists of two connected components and any path $\gamma \in \Gamma$ intersects $\cN$. Thus
$$
\inf_{\gamma \in \Gamma} \sup_{t \in [0,1]} \cJ(\gamma(t)) \geq \inf_{\cN} \cJ.
$$
Since
$$
\inf_{\cN} \cJ = \inf_{u \in E \setminus \{0\}} \sup_{t > 0} \cJ(tu)
$$
there follows, under (J1), that
$$
c := \inf_{\gamma \in \Gamma} \sup_{t \in [0,1]} \cJ(\gamma(t)) = \inf_{\cN} \cJ = \inf_{u \in E \setminus \{0\}} \sup_{t > 0} \cJ(tu) \geq \inf_{\|u\| = r} \cJ(u) > 0.
$$
The existence of a Cerami sequence follows from the mountain pass theorem.
\end{proof}

\begin{Rem}\label{rem:5.2}
From the proof there follows that for any $u \in E \setminus \{0\}$ there is an interval $I_u := [t_{\min} (u), t_{\max} (u)] \subset (0, +\infty)$ such that for any $t \in I_u$ there holds $tu \in \cN$. If $u \in \cN$ one can easily see that $1 \in I_u$. Moreover $I_u$ consists of maxima of $\cJ(tu)$, i.e. $\cJ(t' u) \leq \cJ(tu)$ for $t \in I_u$ and $t' \in (0,+\infty)$, the inequality is strict if $t' \in (0, +\infty) \setminus I_u$.
\end{Rem}

\begin{Rem}
If the inequality in (J3) is strict for $t \neq 1$, then $t_{\min} (u) = t_{\max} (u)$ and the Nehari manifold $\cN$ is homeomorphic to the unit sphere $\cS$ in $E$. Then one can apply the method introduced in \cite{SzW} and obtain the existence of a bounded Palais-Smale sequence (see \cite[Theorem 2.1]{BM-cpaa}). In our case such a homeomorphism does not need to exist and we cannot apply directly the Nehari manifold method. Instead of use the mountain pass theorem one can also obtain the existence of a bounded Palais-Smale sequence using the technique from \cite{dPKrSz}, where (instead of the homeomorphism $\cN \leftrightarrow \cS$) the authors use the set-valued projection $$E \setminus \{0\} \ni u \mapsto \hat{m}(u) := [t_{\min}(u), t_{\max}(u)] u := \{ tu \ : \ t_{\min} (u) \leq t \leq t_{\max}(u) \} \subset \cN$$ and the fact that $\cJ(\hat{m}(u))$ is a real-valued, locally Lipschitz continuous function.
\end{Rem}

\section{Existence and boundedness of Cerami sequences}

We recall the notation $\wt{f} (x,u) := f(x,u) - g(x,u)$ and $\wt{F} (x,u) := F(x,u) - G(x,u)$. Then our functional is of the form
$$
\cJ(u) = \|u\|^2 - \int_{\R^3} F(x,u) \, dx + \int_{\R^3} G(x,u) \, dx = \|u\|^2 - \int_{\R^3} \wt{F} (x,u) \, dx.
$$

Note that \eqref{eq:h} implies that $\wt{f}$ is odd in $u$. Indeed, take any $w \in \R^3 \setminus \{0\}$ and $\alpha \in \R$. Then
$$
\wt{f}(x,\alpha) w = h(x, \alpha w) = h(x, (-\alpha)(-w) ) = -\wt{f}(x,-\alpha)w
$$
and $\wt{f}(x,\alpha) = -\wt{f}(x,-\alpha)$ for any $\alpha \in \R$.

We set
$$
\cI(u) := \int_{\R^3} \wt{F} (x,u) \, dx, \quad u \in X.
$$
Moreover, combining \eqref{f-eps} and \eqref{g-eps} we obtain that for any $\eps > 0$ there is $C_\eps$ such that
\begin{equation}\label{fg-eps}
\left| \wt{f}(x,u) \right| \leq \eps |u| + C_\eps \left( |u|^{q-1} + |u|^{p-1} \right).
\end{equation}

To obtain the existence of a Cerami sequence, we need to verify (J1)--(J3) in Theorem \ref{abstract}.

\begin{enumerate}
\item[(J1)] Observe that, from \eqref{f-eps} and Sobolev embeddings,
$$
\int_{\R^3} \tilde{F} (x,u) \, dx \leq \int_{\R^3} F (x,u) \, dx \leq \varepsilon |u|_2^2 + C_\varepsilon |u|_p^p \leq C \left( \varepsilon \|u\|^2 + C_\varepsilon \|u\|^p \right).
$$
Choosing properly $\varepsilon > 0$ and $r > 0$ we see that
$$
\int_{\R^3} \tilde{F} (x,u) \, dx \leq \frac14 \|u\|^2
$$
for $\|u\| \leq r$. Then
$$
\cJ(u) = \frac12 \|u\|^2 - \int_{\R^3} \tilde{F} (x,u) \, dx \geq \frac14  \|u\|^2 = \frac{r^2}{4}
$$
for $\|u\|=r$.
\item[(J2)] Let $t_n \to +\infty$ and $u_n \to u \neq 0$. Then
\begin{align*}
\frac{\cI(t_n u_n)}{t_n^2} &= t_n^{q-2} \frac{\cI(t_n u_n)}{t_n^q} = t_n^{q-2} \frac{\int_{\R^3} \tilde{F} (x,t_n u_n) \, dx}{t_n^q} \\ &= t_n^{q-2} \left( \int_{\R^3} \frac{F(x,t_n u_n)}{t_n^q} \, dx - \int_{\R^3} \frac{G(x,t_n u_n)}{t_n^q} \, dx \right).
\end{align*}
From (F3) and Fatou's lemma there follows that
$$
\int_{\R^3} \frac{F(x,t_n u_n)}{t_n^q} \, dx \to +\infty.
$$
Hence it is sufficient to show that $\int_{\R^3} \frac{G(x,t_n u_n)}{t_n^q} \, dx$ is bounded from above. Taking \eqref{g-eps} into account, we see that
$$
\int_{\R^3} \frac{G(x,t_n u_n)}{t_n^q} \, dx \leq \varepsilon \frac{|u_n|_2^2}{t_n^{q-2}} + C_\varepsilon |u_n|_q^q = o(1) + C_\varepsilon |u_n|_q^q \leq M
$$
for some constant $M > 0$ and the proof is completed.
\item[(J3)] 
Define 
$$
[0,+\infty) \ni t \mapsto \varphi (t) := \frac{t^2-1}{2} \cI'(u)(u) - \cI(tu) + \cI(u) \in \R.
$$
Note that $\varphi(1) = 0$. (G3) implies that $g(x, tu)u \geq t^{q-1} g(x, u)u$ for $t \in (0,1)$ and $g(x, tu)u \leq t^{q-1} g(x, u)u$ for $t > 1$. Moreover
\begin{align*}
\varphi'(t) &= t \cI'(u)(u) - \cI'(tu)(u) \\ &= \int_{\R^3} f(x,u)tu \, dx - \int_{\R^3} f(x,tu)u \, dx - \int_{\R^3} g(x,u)tu \, dx + \int_{\R^3} g(x,tu)u \, dx.
\end{align*}
Suppose that $t \in (0,1)$. Then
\begin{align*}
\varphi'(t) &\geq \int_{\R^3} f(x,u)tu \, dx - \int_{\R^3} f(x,tu)u \, dx - \int_{\R^3} g(x,u)tu \, dx + \int_{\R^3} t^{q-1} g(x, u)u \, dx.
\end{align*}
The Nehari identity
$$
\| u \|^2 = \int_{\R^3} f(x,u)u - g(x,u)u \, dx
$$
imply that 
$$
\int_{\R^3} f(x,u)u \, dx > \int_{\R^3} g(x,u)u \, dx.
$$
Thus
\begin{align*}
\varphi'(t) &\geq \int_{\R^3} f(x,u)tu \, dx - \int_{\R^3} f(x,tu)u \, dx + (t^{q-1} - t) \int_{\R^3} g(x,u)u \, dx \\
&\geq \int_{\R^3} f(x,u)tu \, dx - \int_{\R^3} f(x,tu)u \, dx + (t^{q-1} - t) \int_{\R^3} f(x,u)u \, dx \\
&= t^{q-1} \int_{\R^3}f(x,u)u - \frac{f(x,tu)u}{t^{q-1}} \, dx 
\end{align*}
Hence $\varphi'(t) \geq 0$ for $t \in (0,1)$, by (F4). Similarly $\varphi'(t) \leq 0$ for $t > 1$ and $\varphi(t) \leq \varphi(1) = 0$ for all $t > 0$.
\end{enumerate}

\begin{Lem}\label{lem:1}
There holds $q F(x,u) \leq f(x,u)u$ for $u \in \R$ and a.e. $x \in \R^3$.
\end{Lem}

\begin{proof}
Suppose that $u > 0$. Note that, under (F4),
$$
F(x,u) = \int_0^u f(x,s) \, ds = \int_0^u \frac{f(x,s)}{s^{q-1}} s^{q-1} \, ds \leq \int_0^u \frac{f(x,u)}{u^{q-1}} s^{q-1} \, ds = \frac{1}{q} \frac{f(x,u)}{u^{q-1}} u^q = \frac1q f(x,u) u.
$$
Similarly for $u < 0$.
\end{proof}

\begin{Lem}
There holds $q G(x,u) \geq g(x,u)u$ for $u \in \R$ and a.e. $x \in \R^3$.
\end{Lem}

\begin{proof}
Take $u > 0$ and compute, as in Lemma \ref{lem:1},
$$
G(x,u) = \int_0^u g(x,s) \, ds = \int_0^u \frac{g(x,s)}{s^{q-1}} s^{q-1} \, ds \geq \int_0^u \frac{g(x,u)}{u^{q-1}} s^{q-1} \, ds = \frac{1}{q} \frac{g(x,u)}{u^{q-1}} u^q = \frac1q g(x,u) u.
$$
Similarly for $u < 0$.
\end{proof}

\begin{Cor}\label{AR-fg}
There holds $q \wt{F}(x,u) \leq \wt{f}(x,u)u$ for $u \in \R$ and a.e. $x \in \R^3$,
\end{Cor}

\begin{Lem}\label{cer-bdd}
Any Cerami sequence $(u_n)$ for $\cJ$ is bounded.
\end{Lem}

\begin{proof}
Observe that $(\cJ(u_n))$ is bounded. Hence, from Corollary \ref{AR-fg}
\begin{align*}
\cJ(u_n) &=\cJ(u_n) - \frac1q \cJ'(u_n)(u_n) + o(1) \\
&= \left( \frac12 - \frac1q \right) \|u_n\|^2 + \frac1q \int_{\R^3}  \wt{f}(x,u_n)u_n - q \wt{F}(x,u_n) \, dx  + o(1) \\
&\geq\left( \frac12 - \frac1q \right)  \|u_n\|^2 + o(1)
\end{align*}
implies that $(u_n)$ is bounded.
\end{proof}

\section{The existence result}

A slight modification of the proof of \cite[Corollary 3.2, Remark 3.3]{Mederski} shows that the following concentration-compactness principle holds true.

\begin{Cor}\label{lem:lions}
Suppose that $(w_n) \subset X$ is bounded and for all $R > 0$ satisfies
\begin{equation}\label{lions}
\lim_{n\to+\infty} \sup_{z \in \R} \int_{B((0,0,z), R)} |w_n|^2 \, dx = 0.
\end{equation}
Then 
$$
\int_{\R^3} |\Psi(x, w_n)| \, dx \to 0 \mbox{ as } n \to +\infty
$$
for any Carath\'eodory function\footnote{We say that $\Psi : \R^3 \times \R \rightarrow \R$ is a Carath\'eodory function if $\Psi = \Psi(x,s)$ is measurable in $x \in \R^3$ and continuous in $s \in \R$.} $\Psi : \R^3 \times \R \rightarrow \R$ satisfying
$$
\lim_{s \to 0} \frac{\Psi(x, s)}{s^2} =  \lim_{|s| \to +\infty} \frac{\Psi(x, s)}{s^6} = 0 \mbox{ uniformly in } x \in \R^3
$$
and $|\Psi(x,s)| \leq c (1 + |s|^6)$ for some $c > 0$.
\end{Cor}

Corollary \ref{lem:lions} easily imply the following fact.

\begin{Lem}\label{lem:7.2}
Suppose that bounded sequence $(w_n) \subset X$ satisfies \eqref{lions} for every $R > 0$. Then
$$
\int_{\R^3} \wt{f}(x, \xi_n) w_n \, dx \to 0
$$
for every bounded $(\xi_n) \subset X$.
\end{Lem}

\begin{proof}
Observe that \eqref{fg-eps} and H\"older's inequality imply that
$$
\left| \int_{\R^3} \wt{f}(x, \xi_n) w_n \, dx \right| \leq \varepsilon | \xi_n |_2 |w_n|_2 + C_\varepsilon |\xi_n|_p^{p-1} |w_n|_p + C_\varepsilon |\xi_n|_q^{q-1} |w_n|_q.
$$
Corollary \ref{lem:lions} imply that $|w_n|_p \to 0$ and $|w_n|_q \to 0$. Since $(\xi_n)$ is bounded in $L^p (\R^3)$ and in $L^q (\R^3)$ we get that
\begin{align*}
\limsup_{n\to+\infty} \left| \int_{\R^3} \wt{f}(x, \xi_n) w_n \, dx \right| \leq \varepsilon \limsup_{n\to+\infty} | \xi_n |_2 |w_n|_2
\end{align*}
for every $\varepsilon > 0$. Therefore
$$
\int_{\R^3} \wt{f}(x, \xi_n) w_n \, dx \to 0,
$$
since $(w_n)$ and $(\xi_n)$ are bounded in $L^2 (\R^3)$.
\end{proof}

\begin{proof}[Proof of Theorem \ref{Th:2}]

Let $(u_n) \subset X$ be a Cerami sequence given by Theorem \ref{abstract}. From Lemma \ref{cer-bdd}, without loss of generality, we may assume that $u_n \weakto u_0$ for some $u_0 \in X$. Assume that for any $R > 0$ \eqref{lions} holds true for $w_n := u_n - u_0$. From the weak-to-weak* continuity of $\cJ'$ there follows that $\cJ'(u_0) = 0$.
Then
\begin{align*}
\cJ'(u_n)(u_n - u_0) &= \| u_n - u_0 \|^2 + \langle u_0, u_n - u_0 \rangle - \int_{\R^3} \wt{f}(x, u_n)(u_n - u_0) \, dx.
\end{align*}
Hence
$$
\| u_n - u_0 \|^2 = \cJ'(u_n)(u_n - u_0) - \langle u_0, u_n - u_0 \rangle + \int_{\R^3} \wt{f}(x, u_n)(u_n - u_0) \, dx.
$$
Since $\cJ'(u_0)(u_n - u_0) = \langle u_0, u_n - u_0 \rangle - \int_{\R^3} \wt{f}(x, u_0)(u_n - u_0) \, dx = 0$ we get
$$
\| u_n - u_0 \|^2 = \cJ'(u_n)(u_n - u_0) + \int_{\R^3} \wt{f}(x, u_n)(u_n - u_0) \, dx - \int_{\R^3} \wt{f}(x, u_0)(u_n - u_0) \, dx.
$$
Obviously
$$
\cJ'(u_n)(u_n - u_0) \to 0.
$$
From Lemma \ref{lem:7.2} there follows that
$$
\int_{\R^3} \wt{f}(x, u_n)(u_n - u_0) \, dx \to 0 \quad \mbox{and} \quad \int_{\R^3} \wt{f}(x, u_0)(u_n - u_0) \, dx \to 0.
$$
Hence $u_n \to u_0$ in $X$. If $u_0 \neq 0$, then we have also that $\cJ(u_n) \to \cJ(u_0) = c$ and the proof is completed. If $u_0 = 0$, we obtain that $\cJ(u_n) \to 0$ and $c = 0$ - a contradiction. 

Thus there are $R > 0$ and $(z_n) \subset \R$ such that
$$
\limsup_{n \to +\infty} \int_{B((0,0,z_n), R)} |u_n|^2 \,dx > 0.
$$
Taking, if necessary, a larger radius $R$ we may assume that $z_n \in \Z$. Moreover, up to a subsequence, $|z_n| \to +\infty$. 

Put $v_n (r, z) := u_n ( r , z - z_n)$. Then $v_n \in X$ (in particular, is cylindircally symmetric), and
$$
\|v_n\| = \|u_n\|, \ \cJ(v_n) = \cJ(u_n) \to c, \ (1+\|v_n\|) \cJ'(v_n) = (1+\|u_n\|) \cJ'(u_n) \to 0,
$$
since $V$, $f$, $g$ are 1-periodic in $x_3$.

Thus $(v_n) \subset X$ is also a bounded Cerami sequence at level $c$ with
$$
\limsup_{n \to +\infty} \int_{B(0, R)} |v_n|^2 \, dx > 0.
$$
Hence, up to a subsequence, $v_n \weakto v_0 \neq 0$ in $X$. Moreover, from the weak-to-weak* continuity there follows that $\cJ'(v_0) = 0$, in particular $v_0 \in \cN$ and $\cJ(v_0) \geq c$. Hence it is sufficient to show that $\cJ(v_0) = c$. 

Indeed, from the weak lower semicontinuity of the norm and the Fatou's lemma
\begin{align*}
c &= \lim_{n\to+\infty} \cJ(v_n) = \lim_{n\to+\infty} \left( \cJ(v_n) - \frac1q \cJ'(v_n) \, dx \right) \\ &= \lim_{n\to+\infty} \left[ \left( \frac12 - \frac1q \right) \| v_n\|^2 + \int_{\R^3}  \frac1q \wt{f}(x,v_n)v_n - \wt{F}(x,v_n) \, dx \right] \\
&\geq \left( \frac12 - \frac1q \right) \| v_0 \|^2  + \int_{\R^3}  \frac1q \wt{f}(x,v_0)v_0 - \wt{F}(x,v_0) \, dx = \cJ(v_0) - \frac1q \cJ'(v_0)(v_0) = \cJ(v_0) \geq c.
\end{align*}
In particular $\cJ(v_0) = c$ and the proof is completed.

\end{proof}

\begin{proof}[Proof of Theorem \ref{Th:3}]
The statement is a direct consequence of Theorem \ref{Th:1} and Theorem \ref{Th:2}.
\end{proof}

\section{Appendix: The multiplicity of solutions}

In the appendix we will show that the problem \eqref{eq:schrodinger} admits infinitely many solutions. Suppose that $u \in X$ is a weak solution to \eqref{eq:schrodinger}. Then, for any $z \in \Z$, $u(\cdot, \cdot, \cdot - z) \in X$ is also a weak solution \eqref{eq:schrodinger}. We introduce the action of $\Z$ on $X$ by
$$
\Z \times X \ni (z, u) \mapsto u(\cdot, \cdot, \cdot - z) \in X.
$$
Let 
$$
\cO(u) := \{ u(\cdot, \cdot, \cdot - z) \ : \ z \in \Z \}
$$
denote the orbit of $u \in X$. We say that two solutions $u, v \in X$ of \eqref{eq:schrodinger} are \textit{geometrically distinct}, if $\cO(u) \cap \cO(v) = \emptyset$. Now we can state the result.

\begin{Th}\label{Th:4}
Suppose that (V), (F1)--(F4), (G1)--(G3) hold. Then there exists infinitely many pairs $\pm u \in H^1 (\R^3)$ of cylindrically symmetric, weak solutions to \eqref{eq:schrodinger} with $\int_{\R^3} \frac{u^2}{r^2} \, dx < +\infty$.
\end{Th} 

As a direct consequence of Theorem \ref{Th:1} we immediately obtain the multiplicity result for the curl-curl problem \eqref{eq:maxwell}.

\begin{Th}\label{Th:5}
Suppose that (V), (F1)--(F4), (G1)--(G3) hold. Then there exist infinitely many pairs $\pm \mathbf{E} \in H^1 (\R^3; \R^3)$ of weak solution to \eqref{eq:maxwell} of the form \eqref{eq:form} for some cylindrically symmetric $\pm u \in H^1 (\R^3)$.
\end{Th}

We will provide the sketch of the proof based on the method introduced in \cite{KrSz}, see also \cite{SzW}. Let $\hat{m} : X \setminus \{0\} \rightarrow 2^\cN$ be the set-valued map defined by
$$
\hat{m}(u) := [t_{\min}(u), t_{\max}(u)] u,
$$
where $t_{\min}(u), t_{\max}(u)$ are positive numbers defined in the proof of Theorem \ref{abstract}, see also Remark \ref{rem:5.2}. From the proof of Theorem \ref{abstract} there follows that $\cJ(v_1) = \cJ(v_2)$ for all $v_1, v_2 \in \hat{m}(u)$ with $u \in X \setminus \{0\}$ fixed. Hence the composition $\cJ \circ \hat{m} : X \setminus \{0\} \rightarrow \R$ is a real-valued map. Let $m := \hat{m} |_S$, where
$$
S := \{ u \in X \ : \ \|u\| = 1 \}.
$$
Then to each $[t_{\min}(u), t_{\max}(u)] u \subset \cN$ corresponds exactly one $u \in S$. We observe the following regularity of $\cJ \circ \hat{m}$, see \cite[Proposition 2.6]{KrSz}.

\begin{Lem}
$\cJ \circ \hat{m}$ is locally Lipschitz continuous.
\end{Lem}

Hence, the Clarke's subdifferential $\partial (\cJ \circ \hat{m})$ is well-defined and we set
$$
\mathscr{C} := \{ u \in S \ : \ 0 \in \partial (\cJ \circ \hat{m})(u) \}.
$$
We recall that $u \in S$ is called a critical point of $\cJ \circ \hat{m}$, if $0 \in \partial (\cJ \circ \hat{m})(u)$. Similarly as in \cite{KrSz} we note that $u \in S$ is a critical point of $\cJ \circ \hat{m}$ if and only if $m(u)$ consists of critical points of $\cJ$. One can also show the similar correspondence between Palais-Smale sequences.

\begin{proof}[Proof of Theorem \ref{Th:4}]
We choose a subset $\cF \subset \mathscr{C}$ such that $\cF = - \cF$ and each orbit in $\mathscr{C}$ has a unique representative in $\cF$. We assume, by the contradiction, that $\cF$ is a finite set. Then, repeating the proof of \cite[Proposition 3.1]{KrSz}, \cite[Lemma 2.14]{SzW} we show the so-called \textit{discreteness of Palais-Smale sequences}. The discreteness of Palais-Smale sequencess, evenness and $\Z$-invariance of $\cJ$ allows us to repeat proofs of \cite[Proposition 3.2, Proposition 3.3]{KrSz}, \cite[Lemma 2.15, Lemma 2.16]{SzW} in our case. Hence, we show that there is an increasing, infinite sequence $c_1 < c_2 < c_3 < \ldots$ of critical values of $\cJ \circ \hat{m}$. It is a contradiction with the fact that $\cF$ is finite.
\end{proof}

\section*{Acknowledgements}
Bartosz Bieganowski was partially supported by the National Science Centre, Poland (Grant No. 2017/25/N/ST1/00531).

\end{document}